\documentclass{amsart}
\usepackage{latexsym, graphicx, amssymb, cite, hyperref, tikz-cd, tikz}
\usepackage[normalem]{ulem}
\usepackage{subfig}
\newtheorem{theorem}{Theorem}[section] 
\newtheorem{lemma}[theorem]{Lemma}

\newtheorem{question}[theorem]{Question}

\theoremstyle{definition}
\newtheorem{definition}[theorem]{Definition}

\theoremstyle{remark}

\numberwithin{equation}{section}

\newcommand{\R}{{\mathbb R}}

\DeclareMathOperator*{\supp}{supp}
\DeclareMathOperator*{\singsupp}{sing\, supp}

\title{Surfaces in which every point sounds the same}

\author{Feng Wang}
\address{School of Mathematical Sciences, Zhejiang University, Hangzhou 310027, PR China}
\email{wfmath@zju.edu.cn}

\author{Emmett L. Wyman}
\address{Department of Mathematics, University of Rochester, Rochester NY}
\email{emmett.wyman@rochester.edu}

\author{Yakun Xi}
\address{School of Mathematical Sciences, Zhejiang University, Hangzhou 310027, PR China}
\email{yakunxi@zju.edu.cn}

\begin{document}

\begin{abstract}
We address a maximally structured case of the question, ``Can you hear your location on a manifold," posed in \cite{echolocation} for dimension $2$. In short, we show that if a compact surface without boundary sounds the same at every point, then the surface has a transitive action by the isometry group. In the process, we show that you can hear your location on Klein bottles and that you can hear the lengths and multiplicities of looping geodesics on compact hyperbolic quotients.
\end{abstract}

\maketitle

\section{Introduction}

\cite{echolocation} poses the following question: If you find yourself standing at some unknown position in a familiar Riemannian manifold, is it possible to deduce one's location, up to symmetry, by clapping your hands once and listening to the reverberations? 

To phrase this question precisely, we require a bit of setup. We let $(M,g)$ be a compact Riemannian manifold. We consider the natural generalization of the Laplacian to $(M,g)$, the Laplace-Beltrami operator, written locally by
\[
    \Delta_g = |g|^{-1/2} \sum_{i,j} \partial_i( g^{ij} |g|^{1/2} \partial_j ).
\]
By the spectral theorem, $L^2(M)$ admits an orthonormal basis of Laplace-Beltrami eigenfunctions, $e_1, e_2, \ldots$ satisfying
\[
    \Delta_g e_j = -\lambda_j^2 e_j.
\]
If the manifold has a boundary, we require some appropriate boundary conditions (e.g. Dirichlet or Neumann) to make the Laplace-Beltrami operator self-adjoint. The \emph{pointwise Weyl counting function} is defined as
\[
    N_x(\lambda) = \sum_{\lambda_j \leq \lambda} |e_j(x)|^2
\]
and is independent of the choice of orthonormal basis. The question above is then phrased as follows.

\begin{question}[\cite{echolocation}] \label{question: echolocation} Take $(M,g)$ as above, and suppose $x, y \in M$. If $N_x(\lambda) = N_y(\lambda)$ for all $\lambda$, must there exist an isometry $M \to M$ mapping $x$ to $y$?
\end{question}

From here on out, we will say that \emph{echolocation holds} for a manifold if the question above is answered in the affirmative for that manifold.

To see how the pointwise counting function $N_x$ is related to the physical interpretation of the problem as described above, we consider the solution operator $\cos(t\sqrt{-\Delta_g})$ taking initial data $f$ to the unique solution to the initial-value initial value problem
\[
    (\Delta_g - \partial_t^2) u = 0 \qquad \text{ and } \qquad \begin{cases}
        u(0) = f \\
        \partial_t u(0) = 0.
    \end{cases}
\]
We recall (or quickly verify) the Schwartz kernel of the solution operator is given in terms of the eigenbasis by
\[
    \cos(t\sqrt{-\Delta_g})(x,y) = \sum_j \cos(t\lambda_j) e_j(x) \overline{e_j(y)}.
\]
Taking a formal cosine transform of the measure $dN_x$ yields
\begin{align*}
    \int_{-\infty}^\infty \cos(t\lambda) \, dN_x(\lambda) = \sum_{j} \cos(t \lambda_j) |e_j(x)|^2 = \cos(t\sqrt{-\Delta_g})(x,x).
\end{align*}
We interpret the rightmost side as the distributional solution $u$ of the wave equation above with initial data $f = \delta_x$, and evaluated at $(t,x)$. Here, the initial $\delta_x$ represents our sharp ``snap" whose reverberations we are listening to at point $x$ for all $t \geq 0$. Furthermore, since $N_x$ is real and supported on $[0,\infty)$, $dN_x$ is uniquely identified by its cosine transform on $[0,\infty)$. Hence, there is no loss of information when passing between interpretations of the problem.

In \cite{echolocation}, we prove that echolocation holds for boundaryless manifolds equipped with generic metrics. Generically, the group of isometries on $M$ is trivial, and so this can be thought of as the minimally structured case. This paper starts to address the maximally structured case. Namely, if $N_x$ is the same function over all $x$ in $M$, must the isometry group act transitively on $M$. In this paper, we answer this question for two-dimensional surfaces without boundary.

\begin{theorem}[Main Theorem] \label{thm: main}
    Let $(M,g)$ be a compact Riemannian surface 
    with empty boundary.  $N_x(\lambda)$ is constant in $x$ for each $\lambda$, if and only if the isometry group acts transitively on $M$.
\end{theorem}
As a byproduct, we also prove the following.
\begin{theorem}[Klein bottle] \label{thm: Klein}
    Let $(M,g)$ be a flat Klein bottle. Then echolocation holds on $M$ despite the fact that $N_x(\lambda)$ is not constant in $x$. 
\end{theorem}

\subsection*{Acknowledgements}Xi was  supported by the National Key Research and Development
Program of China No. 2022YFA1007200
 and NSF China Grant No. 12171424. Wang was supported by NSF China Grant No. 12031017 and Zhejiang Provincial Natural Science Foundation of China under Grant No. LR23A010001. Wyman was supported by NSF grant DMS-2204397. The authors would like to thank Wenshuai Jiang for helpful conversations.

\section{The strategy of proof for Theorem \ref{thm: main}}

 It is clear that if the there is an isometry on $M$ mapping $x$ to $y$ then $N_x$ must equal $N_y$, and thus the necessity direction of Theorem \ref{thm: main} is trivial. We shall prove the remaining direction. We will need to extract two types of information from $N_x$: \emph{local} information in the form of curvature, and \emph{global} information regarding the behavior of geodesics. We begin with the former. By classical pointwise asymptotics for the heat kernel \cite{heatsecondterm}, we have for small $t > 0$,
\begin{multline*}
    \int_{-\infty}^\infty e^{-t\lambda^2} \, dN_x(\lambda) = \sum_j e^{-t\lambda_j^2} |e_j(x)|^2 = e^{t\Delta_g}(x,x) \\
    = \frac{1}{4\pi t} \left( 1 + \frac{t}{3} K(x) + O(t^2) \right),
\end{multline*}
where, since $\dim M = 2$, $K(x)$ is the sectional curvature at $x$. It follows from the hypotheses that $M$ has constant curvature. After rescaling the metric, there are only three cases:
\begin{enumerate}
    \item $K = 1$, where $M$ is the sphere or the projective sphere with the standard metric.
    \item $K = 0$, where $M$ is a flat torus or flat Klein bottle.
    \item $K = -1$, where $M$ is a compact quotient of the hyperbolic plane $\mathbb H$.
\end{enumerate}
Note the conclusion of Theorem \ref{thm: main} holds in case (1) since spheres and real projective spaces are symmetric spaces. It also holds if $M$ is orientable in case (2). We aim to exclude Klein bottles from case (2) and all of case (3).  We will do the former in Section \ref{sec: klein bottle} by direct calculation and the latter in Section \ref{sec: hyperbolic surface} by leveraging the global information we can extract from $N_x$, namely information about geodesic loops at $x$.

\begin{definition}\label{def: loops}
A \emph{geodesic loop} at $x$ is a geodesic segment in $M$ with both endpoints at $x$. Note, a geodesic loop need not close smoothly. The \emph{looping times} at $x$ is the set
\[
    \mathcal L_x := \{ |\gamma| : \gamma \text{ is a looping geodesic at $x$} \}
\]
of lengths of looping geodesics.
\end{definition}

The relationship between looping geodesics and the behavior of the pointwise counting function is very well studied \cite{Saf, SZDuke, STZ, SZRev, CanGal4}. In particular, for fixed $x$,
\[
    \singsupp_t \cos(t\sqrt{-\Delta_g})(x,x) \subset \mathcal L_x \cup -\mathcal L_x.
\]
One can think of the looping times as the times at which you hear an echo after your initial clap. For the sake of Question \ref{question: echolocation}, it would be convenient to show that the inclusion above is actually an equality, but this does not seem to hold in general. 
In the special case where $(M,g)$ has constant curvature $-1$, however, the looping times $\mathcal L_x$ and even the multiplicities of the looping geodesics are audible, as we will show in Lemma \ref{lem: loops are audible}. This will be used to derive the required contradiction to exclude case (3).

\section{Echolocation on a Klein bottle} \label{sec: klein bottle}
In this section, we explicitly compute the pointwest Weyl counting function $N_x(\lambda)$ on a Klein bottle, and prove Theorem \ref{thm: Klein}.

We shall follow the definition of a Klein bottle $\mathbb K_{a,b}$ in \cite{Klein}. A point in $\mathbb K_{a,b}$ is identified with a point $x=(x_1,x_2)$ in the rectangle $[0,a/2]\times[0,b]$, with its horizontal sides identified with the same orientation and the vertical sides identified with the opposite orientations. It is shown in \cite{Klein} that a complete family of real eigenfunctions of $\mathbb K_{a,b}$ is given by the following functions.
\begin{equation}
    \begin{cases}
    \cos\Big(\dfrac{2\pi nx_2}{b}\Big), &\text{for }m=0,\ n\in\mathbb N,\\
    \cos\Big(\dfrac{2\pi mx_1}{a}\Big)\cos\Big(\dfrac{2\pi nx_2}{b}\Big),\ \sin\Big(\dfrac{2\pi mx_1}{a}\Big)\cos\Big(\dfrac{2\pi nx_2}{b}\Big), &\text{for even }m\in\mathbb Z^+,\ n\in\mathbb N,\\
      \cos\Big(\dfrac{2\pi mx_1}{a}\Big)\sin\Big(\dfrac{2\pi nx_2}{b}\Big),\ \sin\Big(\dfrac{2\pi mx_1}{a}\Big)\sin\Big(\dfrac{2\pi nx_2}{b}\Big), &\text{for odd }m\in\mathbb Z^+,\ n\in\mathbb \mathbb Z^+.
    \end{cases}
\end{equation}
After normalization, we obtain a orthonormal basis of eigenfunctions, $e_{\lambda_{m,n}}$, with eigenvalues $\lambda_{m,n}:=2\pi\sqrt{m^2/a^2+n^2/b^2}$. We record the following table for the norm squared of each eigenfunction.
\begin{table}
\begin{tabular}{|c|c|c|}
\hline
\textbf{$(m,n)$} & \textbf{$|e_{\lambda_{m,n}}|^2(x)$} & \textbf{$\lambda_{m,n}$} \\
\hline
$(2k,0),k\in\mathbb N$ & $\frac{2}{ab}$ & $2\pi m/a$ \\
$(2k,n),(k,n)\in\mathbb N\times\mathbb Z^+$ & $\frac{4}{ab}\cos^2\big(\frac{2\pi nx_2}{b}\big)$ & $2\pi\sqrt{m^2/a^2+n^2/b^2}$ \\
$(2k+1,n),(k,n)\in\mathbb N\times\mathbb Z^+$ & $\frac{4}{ab}\sin^2\big(\frac{2\pi nx_2}{b}\big)$ & $2\pi\sqrt{m^2/a^2+n^2/b^2}$ \\
\hline
\end{tabular}
\caption{$|e_{\lambda_{m,n}}|^2(x)$ for each $(m,n)\in\mathbb N^2$.}
\end{table}

first suppose $1/b< 2/a$. Then we have \[\sum_{\lambda_{m,n}=2\pi/b}|e_{\lambda_{m,n}}|^2(x)=\dfrac{4}{ab}\cos^2\Big(\dfrac{2\pi x_2}{b}\Big),\]
which is not constant in $x_2$, and thus $N_x$ is not constant in $x$ in this case. For the remaining case when $1/b\ge 2/a$, 

It is easy to see that $N_x(\lambda)$ cannot be constant in $x$. In fact, since $0< 1/b<\sqrt{1/a^2+1/b^2}.$ The multiplicity of $\lambda_{m,n}=2\pi/b$ will be one, unless $1/b=2\ell/a$ for some $\ell\in\mathbb Z^+$.  Therefore we must have 
\[
\begin{cases}
\sum\limits_{\lambda_{m,n}=2\pi/b}|e_{\lambda_{m,n}}|^2(x)=\dfrac{4}{ab}\cos^2\Big(\dfrac{2\pi x_2}{b}\Big)+\dfrac{2}{ab}, &\text{if $1/b=2\ell/a$ for some $\ell\in\mathbb Z^+$},\\
\sum\limits_{\lambda_{m,n}=2\pi/b}|e_{\lambda_{m,n}}|^2(x)=\dfrac{4}{ab}\cos^2\Big(\dfrac{2\pi x_2}{b}\Big), &\text{otherwise}.
\end{cases}
\]
In either case, $N_x$ is not constant in $x$. Furthermore, we observe that any point $x\in\mathbb K_{a,b}$ can be mapped, via a self-isometry of $\mathbb K_{a,b}$, to a point $\tilde x\in \{0\}\times[0,b/4]$. To see this, we first notice that any point can be associated with a point on $\tilde x\in \{0\}\times[0,b]$ via horizontal translations. Now we claim that if we divide $\mathbb K_{a,b}$ evenly into four  horizontal strips, say $A, B, C,$ and $D$, then these strips can be mapped to one another via suitable isometries. 
 \begin{figure}
    \centering
    \subfloat{{\includegraphics[width=5cm]{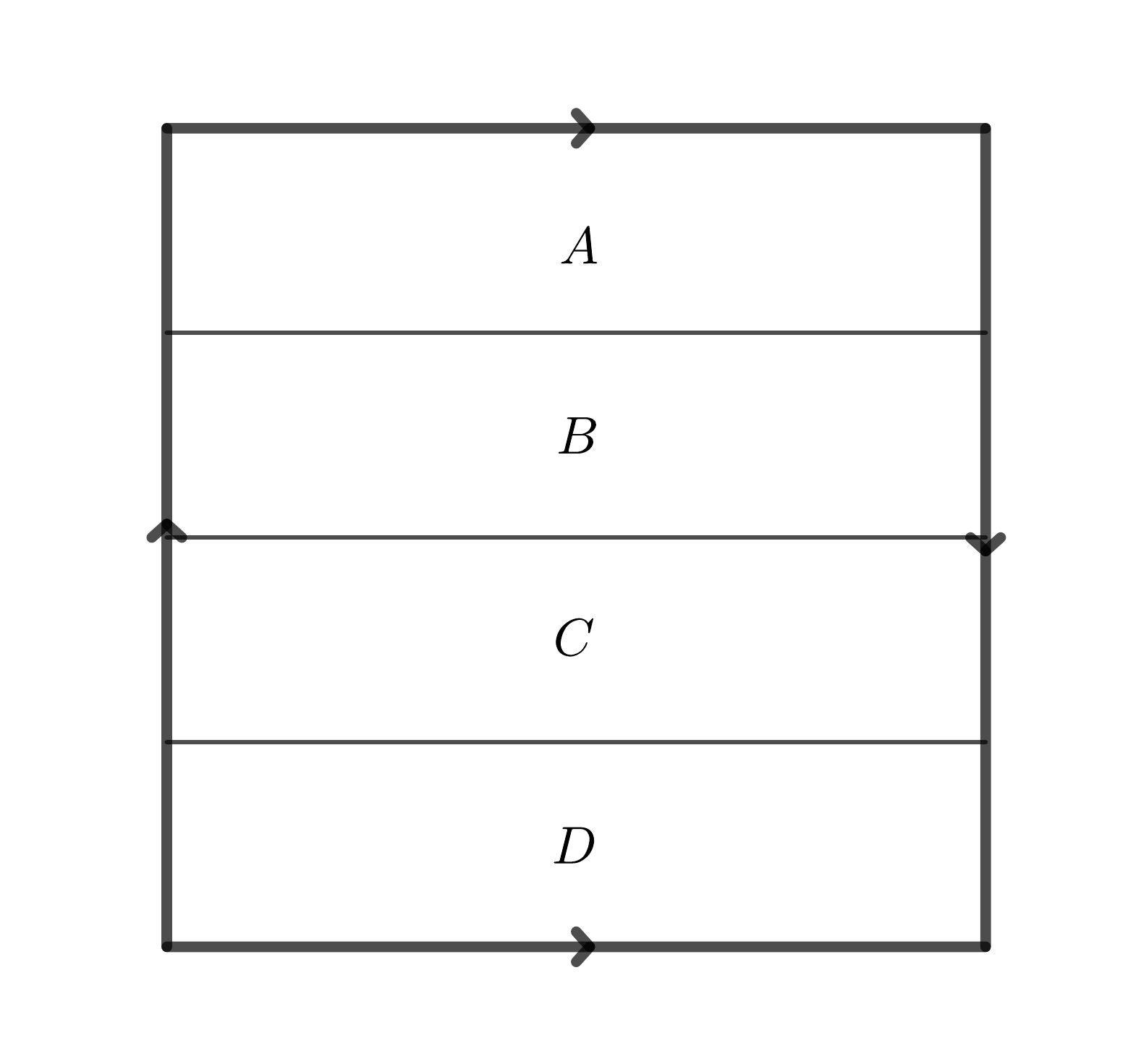} }}
    \qquad
    \subfloat{{\includegraphics[width=5cm]{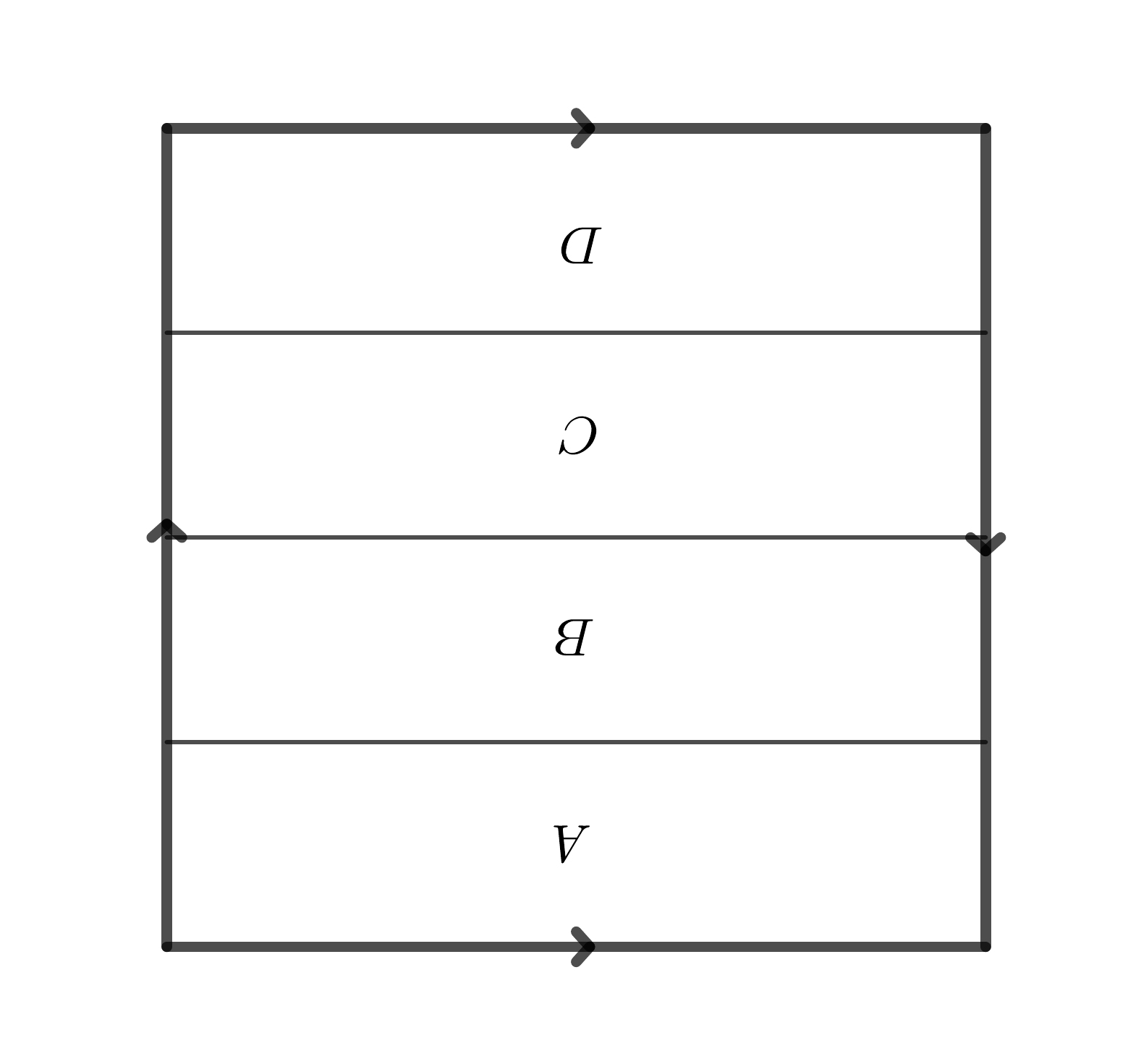} }}
    \subfloat{{\includegraphics[width=5cm]{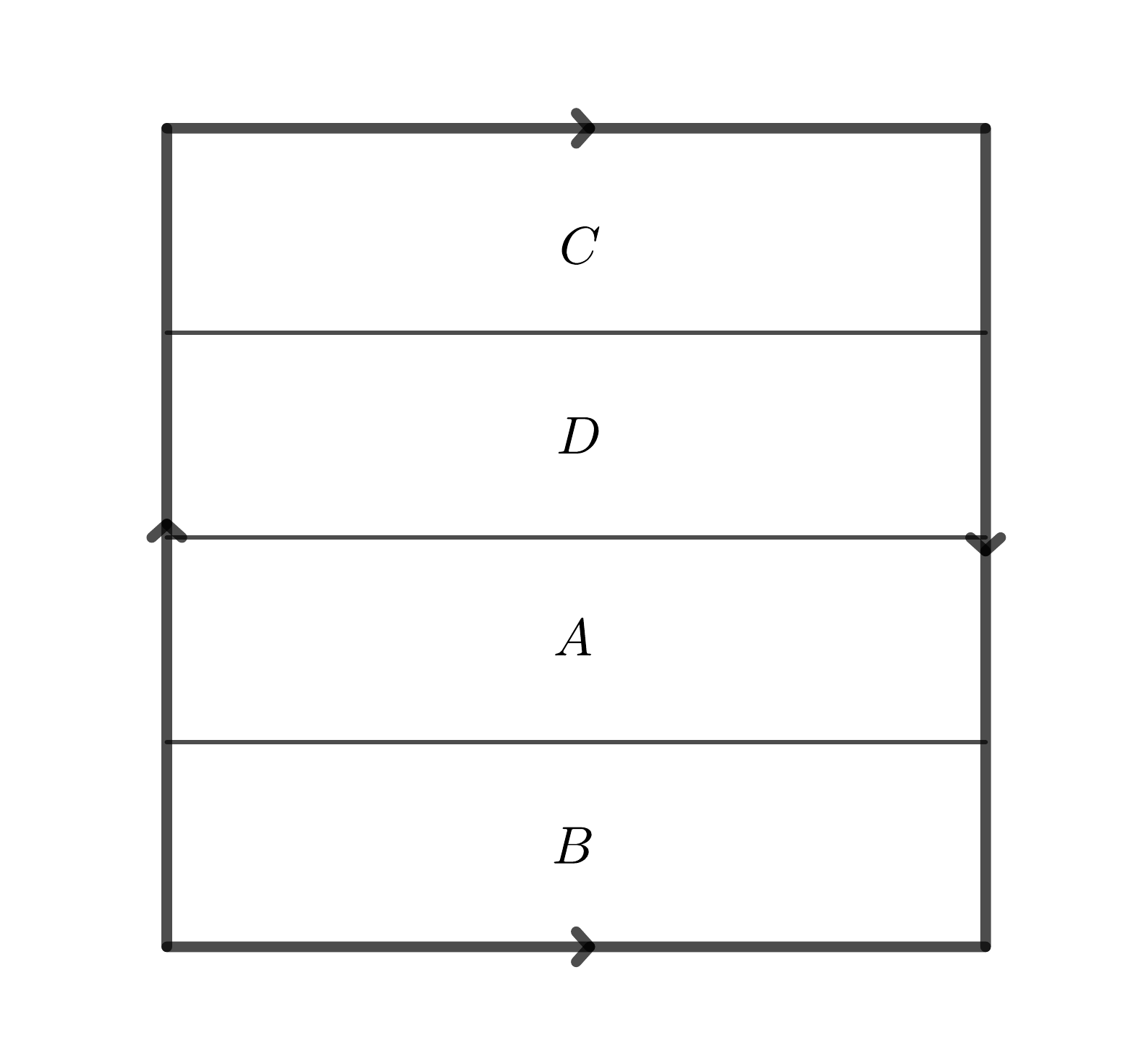} }}
    \caption{Four identical pieces on a Klein bottle.}
    \label{fig: Klein bottle}
\end{figure}
Indeed, since our representation of $\mathbb K_{a,b}$ is centrosymmetric, we see that $A$ can be mapped to $D$ and $B$ can be mapped to $C$ if we turn $\mathbb K_{a,b}$ 180 degree about its center. Next, we notice that if we cut  $\mathbb K_{a,b}$ along the middle horizontal line $y=b/2$, 
 and then patch $A$ and $D$ together  again to recover  $\mathbb K_{a,b}$. See Figure \ref{fig: Klein bottle}. This shows that $A$ can be mapped to $C$ and $B$ can be mapped to $D$, and our claim is proved.  Since $\cos^2\Big(\dfrac{2\pi x_2}{b}\Big)$ is always audible, we conclude that echolocation holds on $\mathbb K_{a,b}$.

\section{Excluding compact hyperbolic surfaces} \label{sec: hyperbolic surface}

Consider the audible distribution
\[
    \cos(t \sqrt{-\Delta_g})(x,x) = \int_{-\infty}^\infty \cos(t\lambda) \, dN_x(\lambda)
\]
in $t$ as described in the introduction. In \cite{Berard}, B\'erard lifts the cosine wave kernel to the universal cover $(\R^2, \tilde g)$ by the identity
\[
    \cos(t \sqrt{-\Delta_g})(x,x) = \sum_{\gamma \in \Gamma} \cos(t \sqrt{-\Delta_{\tilde g}})(\tilde x, \gamma(\tilde x)),
\]
where here $\Gamma$ is the deck group and $\tilde x$ is a lift of $x$. The fact that there are no conjugate pairs on $(\R^2, \tilde g)$ allows us B\'erard to use Hadamard's parametrix globally. We will do the same.

Before writing down the parametrix, we recall some standard notions. We say a smooth function $a$ on $\R^n \times \R^N$ is a \emph{symbol} of order $m$ if it satisfies, for each compact $K \subset \R^n$ and multiindices $\alpha$ and $\beta$,
\[
    \sup_{x \in K} |\partial_\theta^\alpha \partial_x^\beta a(x,\theta)| \leq C_{K,\alpha,\beta} (1 + |\theta|)^{m - |\alpha|}.
\]
We denote the set of symbols of order $m$ on $\R^n \times \R^N$ as $S^m(\R^n \times \R^N)$. For more on symbols and symbol classes, see e.g. \cite{HormanderPaper, HIII, DuistermaatFIOs, SFIO}.

We will summarize what we need of the Hadamard parametrix from \cite{Hang}. We write
\[
    \cos(t\sqrt{-\Delta_{\tilde g}})(\tilde x, \tilde y) = K_N(t, \tilde x, \tilde y) + R_N(t, \tilde x, \tilde y)
\]
where we will characterize $K_N$ as an oscillatory integral and $R_N$ can be made to be quite smooth. Combining Remark 1.2.5, (3.6.10), and (5.2.16) of \cite{Hang}, we have
\[
    K_N(t,\tilde x, \tilde y) = (2\pi)^{-2} \alpha_0(\tilde x, \tilde y) \int_{\R^2} e^{i \varphi(\tilde x, \tilde y, \xi) \pm i tp(\tilde y, \xi)} a_\pm(t, \tilde x, \tilde y, \xi) \, d\xi
\]
where
\[
    a_\pm(t,\tilde x, \tilde y, \xi) - \frac 12 \alpha_0(\tilde x, \tilde y) \in S^{-2}(\R^{1 + 2 + 2} \times \R^2)
\]
for some appropriate phase function $\varphi(\tilde x, \tilde y, \xi)$ and an appropriate smooth function $\alpha_0(\tilde x, \tilde y)$, which we will now describe.

Since $\tilde M$ is nonpositively curved, the exponential map $\exp_{\tilde y} : T_{\tilde y} \tilde M \to \tilde M$ is a diffeomorphism. We use the logarithm $\log_{\tilde y} : \tilde M \to T_{\tilde y} \tilde M$ to denote its inverse. Remark 1.2.5 of \cite{Hang} characterizes the phase function above as
\[
    \varphi(\tilde x, \tilde y, \xi) = \langle \log_{\tilde y}(\tilde x), \xi \rangle.
\]
The leading coefficient $\alpha_0$ is characterized as
\[
    \alpha_0(\tilde x, \tilde y) = |g(\log_{\tilde y}(\tilde x))|^{-1/4},
\]
where the metric $g$ is that of geodesic normal coordinates about $\tilde y$. Finally, the remainder term $R_N$ is $C^{N - 5}$ in $(t,\tilde x, \tilde y)$ by the discussion after (5.2.18). Furthermore, by Huygen's principle, the remainder term $R_N$ can be made to be supported in $d_{\tilde g}(\tilde x, \tilde y) \leq 2|t|$.

We will use the parametrix above to establish the following key asymptotic quantity:

\begin{lemma}\label{lem: pre-trace}
Let $(M,g)$ be a boundary-less Riemannian surface with nonpositive sectional curvature. Let $\chi$ be a Schwartz-class function on $\R$ with $\widehat \chi$ supported in a compact subset of $\mathbb R^+$. 
    \begin{multline*}
        \int_{-\infty}^\infty e^{-it \lambda} \widehat \chi(t) \cos(t\sqrt{-\Delta_g})(x,x) \, dt \\
        = (2\pi)^{-1/2} \lambda^{1/2} \sum_{\gamma \in \Gamma} e^{\pi i / 4} e^{-i \lambda d_{\tilde g}(\tilde x, \gamma(\tilde x))} \frac{\alpha(\tilde x, \gamma(\tilde x))}{d_{\tilde g}(\tilde x, \gamma(\tilde x))^{1/2}} \widehat \chi(d_{\tilde g}(\tilde x, \gamma(\tilde x))) + O(\lambda^{-1/2}).
    \end{multline*}
\end{lemma}

\begin{proof}
    By lifting to the universal cover and using Hadamard's parametrix above, we write
    \begin{align*}
        &\int_{-\infty}^\infty e^{-it \lambda} \widehat \chi(t) \cos(t\sqrt{-\Delta_g})(x,x) \, dt \\
        &= \sum_{\gamma \in \Gamma} \int_{-\infty}^\infty e^{-it\lambda} \widehat \chi(t) \cos(t\sqrt{-\Delta_{\tilde g}})(\tilde x, \gamma(\tilde x)) \, dt \\
        &= I + II
    \end{align*}
    where we have main term
    \[
        I = (2\pi)^{-2} \sum_{\gamma \in \Gamma} \sum_\pm \int_{-\infty}^\infty \int_{\R^2} e^{i\langle \log_{\tilde y}(\tilde x), \xi \rangle \pm it|\xi| - it\lambda} \widehat \chi(t) a_\pm(t, \tilde x, \gamma(\tilde x), \xi) \, d\xi \, dt
    \]
    and remainder term
    \[
        II = \sum_{\gamma \in \Gamma} \sum_\pm \int_{-\infty}^\infty e^{-it\lambda} \widehat \chi(t) R_N(t, \tilde x, \gamma(\tilde x)) \, dt.
    \]
    Note, term $II$ can be made to decay in $\lambda$ of arbitrary polynomial order by taking $N$ large enough and integrating by parts in $t$. Hence, it contributes nothing to the main term in the lemma, and we turn our attention to $I$. We first note that if the `$\pm$' sign in the exponent is negative, then integration by parts in $t$ yields a rapidly-decaying term which we also neglect. Up to negligible terms, we have
    \[
        I = (2\pi)^{-2} \sum_{\gamma \in \Gamma} \int_{-\infty}^\infty \int_{\R^2} e^{i\langle \log_{\tilde x}(\gamma(\tilde x)), \xi \rangle + it(|\xi| - \lambda)} \widehat \chi(t) a_\pm(t, \tilde x, \gamma(\tilde x), \xi) \, d\xi \, dt.
    \]
    We perform a change of variables $\xi \mapsto \lambda \xi$ and write this term as
    \[
        = (2\pi)^{-2} \lambda^2 \sum_{\gamma \in \Gamma} \int_{-\infty}^\infty \int_{\R^2} e^{i\lambda ( \langle \log_{\tilde x}(\gamma(\tilde x)), \xi \rangle + t(|\xi| - 1))} \widehat \chi(t) a_\pm(t, \tilde x, \gamma(\tilde x), \lambda \xi) \, d\xi \, dt.
    \]
    Let $\beta$ be a smooth bump function with compact support in $(1/2,2)$ taking the value $1$ on a neighborhood of $1$. We cut the integral into $\beta(|\xi|)$ and $1 - \beta(|\xi|)$ parts, the latter of which 
    decays rapidly by integrating by parts in $t$. We are left with
    \[
        (2\pi)^{-2} \lambda^2 \sum_{\gamma \in \Gamma} \int_{-\infty}^\infty \int_{\R^2} e^{i\lambda ( \langle \log_{\tilde x}(\gamma(\tilde x)), \xi \rangle + t(|\xi| - 1))} \widehat \chi(t) \beta(|\xi|) a_\pm(t, \tilde x, \gamma(\tilde x), \lambda \xi) \, d\xi \, dt.
    \]

    Next, we write $\xi = r (\cos \theta, \sin \theta)$ in polar form with $r > 0$ and rephrase the integral as
    \begin{multline*}
        (2\pi)^{-2} \lambda^2 \sum_{\gamma \in \Gamma} \int_{-\infty}^\infty \int_{-\pi}^\pi \int_0^\infty e^{i\lambda ( r (v_1 \cos \theta + v_2 \sin \theta) + t(r - 1))} \\
        \widehat \chi(t) \beta(r) a_\pm(t, \tilde x, \gamma(\tilde x), \lambda r(\cos \theta, \sin \theta)) r \, dr \, d\theta \, dt.
    \end{multline*}
  Note, the phase function can be written
    \[
        \varphi = r (v_1 \cos \theta + v_2 \sin \theta) + t(r - 1)
    \]
    where we take as shorthand 
    $v = \log_{\tilde x}(\gamma(\tilde x))$.    We now employ the method of stationary phase in variables $t$ and $r$. After a rotation, assume without loss of generality at this point that $v = d_{\tilde g}(\tilde x, \gamma(\tilde x)) e_1$. Then, we have
    \[
        \nabla_{t,\theta, r} \varphi = \begin{bmatrix}
            r - 1 \\
            -r d_{\tilde g}(\tilde x, \gamma(\tilde x)) \sin \theta \\
            d_{\tilde g}(\tilde x, \gamma(\tilde x)) \cos \theta + t
        \end{bmatrix}
    \]
    from which we obtain a critical point at $(t, \cos \theta, r) = (\mp d_{\tilde g}(\tilde x, \gamma(\tilde x)), \pm 1, 1)$. Note, for $t \in \supp \widehat \chi$, we require that $\pm = -$ and $\mp = +$. At this sole critical point, the Hessian of the phase reads
    \[
        \nabla_{t,\theta, r}^2 \varphi = 
        \begin{bmatrix}
            0 & 0 & 1 \\
            0 & - d_{\tilde g} & 0 \\
            1 & 0 & 0
        \end{bmatrix},
    \]
    which has determinant and signature
    \[
        |\det \nabla^2_{t, \theta, r} \varphi| = d_{\tilde g}(\tilde x, \gamma(\tilde x)) \quad \text{ and } \quad \operatorname{sig} \nabla^2_{t, \theta, r} \varphi = 1. 
    \]
    Hence, by the method of stationary phase, we have
    \[
        I = (2\pi)^{-1/2} \lambda^{1/2} \sum_{\gamma \in \Gamma} e^{\pi i / 4} e^{-i \lambda d_{\tilde g}(\tilde x, \gamma(\tilde x))} \frac{\alpha(\tilde x, \gamma(\tilde x))}{d_{\tilde g}(\tilde x, \gamma(\tilde x))^{1/2}} \widehat \chi(d_{\tilde g}(\tilde x, \gamma(\tilde x))) + O(\lambda^{-1/2}).
    \]
    The lemma follows.
\end{proof}

As a corollary, we can conclude that the looping times, with multiplicity, are audible for hyperbolic surfaces. To state this precisely, fix $x$ and a lift $\tilde x$ to the hyperbolic plane $\mathbb H$ via the covering map. Then, every looping geodesic at $x$ lifts to the unique geodesic in the universal cover $\mathbb H$ with endpoints at $\tilde x$ and $\gamma(\tilde x)$ where $\gamma$ is an element of the deck group. We conclude that the looping times at $x$ are given by
\[
    \mathcal L_x := \{ d_{\tilde g}(\tilde x, \gamma(\tilde x)) : \gamma \in \Gamma \setminus I\}.
\]
Given any $r$ in this set, we have a multiplicity
\[
    m_x(r) = \#\{ \gamma \in \Gamma \setminus I : d_{\tilde g}(\tilde x, \gamma(\tilde x)) = r\}.
\]
We now have:

\begin{lemma}\label{lem: loops are audible}
    If $(M,g)$ is a compact hyperbolic surface, $m_x$ and $\mathcal L_x$ are audible.
\end{lemma}

\begin{proof} It suffices to show $m_x$ is audible as an integer-valued function on $(0,\infty)$. We can extract this information from the result of Lemma \ref{lem: pre-trace}, but first we specify some of the constants.

Recall from the discussion at the start of the section that $\alpha(\tilde x, \tilde y) = |g(\log_{\tilde x}(\tilde y))|^{-1/4}$ in geodesic normal coordinates. In the case of constant curvature $-1$, we have 
\[
    \alpha(\tilde x, \tilde y) = \Big(\frac{\sinh r}{r}\Big)^{-1/2} 
\]
where $r = d_{\tilde g}(\tilde x, \tilde y)$. Hence,
\begin{multline*}
    \int_{-\infty}^\infty e^{-it \lambda} \widehat \chi(t) \cos(t\sqrt{-\Delta_g})(x,x) \, dt \\
    = (2\pi)^{-1/2} \lambda^{1/2} \sum_{\gamma \in \Gamma} e^{\pi i / 4} e^{-i \lambda d_{\tilde g}(\tilde x, \gamma(\tilde x))} \frac{1}{\sqrt{\sinh d_{\tilde g}(\tilde x, \gamma(\tilde x))}} \widehat \chi(d_{\tilde g}(\tilde x, \gamma(\tilde x))) + O(\lambda^{-1/2}).
\end{multline*}
Now take $\rho$ to be supported on $[-1,1]$ with $\rho(0) = 1$. Then, for $r, \epsilon > 0$ fixed, set
\[
    \widehat \chi_{r,\epsilon}(t) = \sqrt{\sinh t} \  \rho(\epsilon^{-1}(t - r)).
\]
$m_x(r)$ is given by the output of the expression
\[
    \lim_{\epsilon \to 0} \lim_{\lambda \to \infty} (2\pi)^{1/2} e^{-\pi i / 4 + i\lambda r} \lambda^{-1/2} \int_{-\infty}^\infty e^{-it\lambda} \widehat \chi_{r,\epsilon}(t) \cos(t\sqrt{-\Delta_g})(x,x) \, dt.
\]
Note, no matter what $r$ is, the inner limit converges for all sufficiently small $\epsilon > 0$. The lemma is proved.
\end{proof}

Now we assume that a compact hyperbolic manifold has constant pointwise counting function and derive a contradiction. For each $x \in M$, let $r_x$ denote the length of the shortest looping geodesic, i.e. $r_x = \inf \mathcal L_x$. This quantity is audible, and hence is constant. We claim that every point $x$ lies in a closed geodesic in $M$, and we will use this claim to derive a contradiction. We recall the following standard result:

\begin{theorem}[\cite{Peter}, Chap 6] Each homotopy class of loops in a negatively curved manifold has a unique length-minimizing curve, and that curve is a closed geodesic.
\end{theorem}

Suppose we have a geodesic loop through $x$ with length $r_x$. This loop is homotopic to a unique closed geodesic which also has length $r_x$. Since this is the length of our original loop at $x$, our geodesic loop must have been closed.

Absurdities abound already, but here is a straightforward one. There are only countably many closed geodesics on $M$---one for each element of the homotopy classes of loops---but somehow every point in $M$ lies on one.

\bibliography{references}{} 
\bibliographystyle{alpha}

\end{document}